\documentclass[a4paper,11pt]{article}
\usepackage{amsmath,amssymb,amsthm,graphics,latexsym,amsfonts,indentfirst}
\usepackage{fancyhdr}
\usepackage[nooneline,flushleft]{caption2}
\usepackage{dsfont}
\usepackage{amssymb}
\usepackage{subeqnarray,CJK}
\usepackage{bbding,booktabs}
\usepackage{graphicx,color}
\usepackage[a4paper,top=1in, bottom=20ex, left=25mm,
right=25mm,headheight=25pt,headsep=6pt]{geometry}
\newtheorem{Theorem}{\quad Theorem}[section]

\numberwithin{equation}{section}

\title{ \Large Finite difference/local discontinuous Galerkin method for solving the fractional diffusion-wave equation }
\author{{  Leilei Wei\footnote{Corresponding author. E-mail addresses:
leileiwei09@gmail.com.}
}\\
 \footnotesize  \emph{ College of Science, Henan University of Technology, Zhengzhou, Henan 450001, P.R. China}\\
 }
\date{}

\usepackage{indentfirst}
\begin{document}
\maketitle

\newtheorem{lem}{Lemma}[section]
\newtheorem{thm}[lem]{Theorem}
\newtheorem{cor}[lem]{Corollary}
\newtheorem{remark}[lem]{Remark}
\newtheorem{pro}[lem]{Proposition}
\newtheorem{Def}[lem]{Definition}
\newtheorem{Alg}[lem]{Algorithm}

\theoremstyle{plain} \CJKtilde
\newcommand{\D}{\displaystyle}
\newcommand{\DF}[2]{\D\frac{#1}{#2}}
\par
{\small \noindent{\bfseries Abstract:} In this paper a finite difference/local discontinuous Galerkin method for the fractional diffusion-wave equation is presented and analyzed. We first propose a new finite difference method to approximate the time fractional derivatives, and give a semidiscrete scheme in time with the truncation error $O((\Delta t)^2)$, where $\Delta t$ is the time step size. Further we develop a fully discrete scheme for the fractional diffusion-wave equation, and prove that the method is unconditionally stable and convergent with order $O(h^{k+1}+(\Delta t)^{2})$, where $k$ is the degree of piecewise polynomial. Extensive numerical examples are carried out to confirm the theoretical convergence rates.


\par
{\noindent\bfseries Mathematics Subject Classi?cation}: 65M12; 65M06; 35S10}
\section{Introduction}

Fractional calculus, which might be considered as an extension of classical calculus, attracts much attention in recent decades. Fractional order partial differential equations (FPDEs) have been frequently used to solve many scientific problems in various fields, such as quantitative finance, engineering, biology, chemistry, hydrology, and so on \cite{op07,rh00,ks06,rg08,IP99,jan1,jan2}.

However, analytical solutions for the majority of fractional partial differential equations, which are too complex and cannot expressed explicitly, are very difficult to be applied in the science and engineering, so it is a good choice to use numerical methods to finding numerical solutions for fractional partial differential equations, and has very important theoretical and practical significance. The existed methods solving the FPDEs include finite difference methods \cite{fd1,fd2,fd3,DC10,fd4,fd5,fd6,fd7,fd8,fd9,fd10,fd11,fd12,fd13}, finite element methods\cite{fe5,fe1,fe6,fe7,fe4,fe2,fe3}, spectral methods\cite{sp1,sp2,sp3}, discontinuous Gakerkin methods \cite{wei1,wei2}, homotopy perturbation method and the variational method \cite{ot5,ot1,ot2,ot3,ot4,ss07,zhang}.

In this paper we consider the following fractional diffusion-wave
equation
\begin{equation}\label{question}
\begin{split}
&\frac{\partial^{\alpha} u(x,t)}{\partial
t^{\alpha}}-\frac{\partial^{2}u(x,t)}{\partial x^{2}}=f(x,t),
~~~~~~(x,t)\in [a,b]\times [0,T],\\
&u(x,0)=u_0(x),~~~~\frac{\partial u(x,0)}{\partial
t}=u_1(x),~~x\in [a,b], \\
\end{split}
\end{equation}
where $1<\alpha< 2$ is a parameter describing the order of the
fractional time, $f, u_0, u_1$ are given smooth functions. We do not
pay attention to boundary condition in this paper; hence the
solution is considered to be either periodic or compactly supported.

The time fractional derivative in the equation (\ref{question}),
uses the Caputo fractional partial derivative of order $\alpha$,
defined as \cite{DC10}
\begin{equation}\label{time}\frac{\partial^{\alpha}u(x,t)}{\partial
t^{\alpha}}= \frac{1}{\Gamma(2-\alpha)}\int_0^t\frac{\partial^2
u(x,s)}{\partial s^2}\frac{ds}{(t-s)^{\alpha-1}},~t>0,~1<\alpha<2,
\end{equation}
where $\Gamma(\cdot)$ is the Gamma function.

The fractional diffusion-wave equation is obtained by replacing the
first- or second-order time derivative of the classical diffusion or wave equation with
a fractional derivative of order $1<\alpha< 2$, and can be used to interpolate the diffusion
equation and wave equation and model many of the
mechanical responses and acoustics accurately.

The rest of this paper is constructed as follows. In the section 2
some basic notations and theoretic results are introduced. Then in section 3 we construct
our finite difference/discontinuous Galerkin method for the fractional diffusion-wave equation, and stability and error analysis are given.  Numerical results are presented in section 4, and the concluding remarks is included in the
final section.

\section{Notations and auxiliary results}

In this section we introduce some notations and definitions that will be used later in
the following sections.

Let $\Omega=[a,b]$ be a finite domain, and a partition is given by
$$a=x_{\frac{1}{2}}<x_{\frac{3}{2}}<\cdots<x_{N+\frac{1}{2}}=b,$$
we denote the cell by $I_{j}=[x_{j-\frac{1}{2}},
x_{j+\frac{1}{2}}],$ for $j=1,\cdots N,$  and
the cell lengths $\Delta
x_j=x_{j+\frac{1}{2}}-x_{j-\frac{1}{2}},~1\leq j\leq N,$
$h=\max\limits_{1\leq j\leq N}\Delta x_j$.

We denote by $u_{j+\frac{1}{2}}^+$ and $u_{j+\frac{1}{2}}^-$ the
values of $u$ at $x_{j+1/2}$, from the right cell $I_{j+1}$ and from
the left cell $I_j$,respectively.

The piecewise-polynomial space $V_h^k$ is defined as the space of
polynomials of the degree up to $k$ in each cell $I_j$, i.e.
\begin{equation}
V_h^k=\{v:v\in P^k(I_j),x\in I_j,j=1,2,\cdots N\}\nonumber.
\end{equation}

For error estimates, we will be using two projections in one
dimension $[a,b]$, denoted by $\mathcal{P}$, i.e., for each $j$,
\begin{eqnarray}\label{pro1}
&&\int_{I_j}(\mathcal{P}\omega(x)-\omega(x))v(x)=0,\forall v\in
P^k(I_j),
\end{eqnarray}

and special projection $\mathcal{P}^\pm$, i.e., for each $j$,
\begin{eqnarray}\label{pro2}
\slabel{eq0}
&&\int_{I_j}(\mathcal{P}^+\omega(x)-\omega(x))v(x)=0,\forall v\in
P^{k-1}(I_j),\nonumber\\
 \mbox{and}
 \slabel{eq1}
&&\mathcal{P}^+\omega(x_{j-\frac{1}{2}}^+)=\omega(x_{j-\frac{1}{2}})\nonumber\\
\slabel{eq2}
&&\int_{I_j}(\mathcal{P}^-\omega(x)-\omega(x))v(x)=0,\forall v\in
P^{k-1}(I_j), \nonumber\\
\mbox{and}\slabel{eq3}
&&\mathcal{P}^-\omega(x_{j+\frac{1}{2}}^-)=\omega(x_{j+\frac{1}{2}}).
\end{eqnarray}

For the above projections $\mathcal{P}$ and $\mathcal{P}^\pm$, we have \cite{cs982,lx,xia,xushu08}
\begin{eqnarray}\label{projection111}
&&\|\omega^e\|+h\|\omega^e\|_{\infty}+h^{\frac{1}{2}}\|\omega^e\|_{\tau_h}\leq
Ch^{k+1},
\end{eqnarray}
where $\omega^e=\mathcal{P}\omega-\omega$ or
$\omega^e=\mathcal{P}^\pm\omega-\omega$.

The notations are used: the scalar inner product
on $L^2(D)$ be denoted by $(\cdot,\cdot)_{D}$, and the associated
norm by $\|\cdot\|_{D}$. If $D=\Omega$, we drop $D$.
In the present paper we use $C$ to denote a positive constant which
may have a different value in each occurrence.

\section{The schemes}
In this section, we first present a finite difference method to approximate the time fractional derivatives, and then give the implicit fully discrete scheme with space discretized by the local discontinuous Galerkin method. Stability and convergence are detailed analysis.
\subsection{Time fractional derivative discretization}
We divide the interval $[0,T]$ uniformly with a time step size $\Delta t=T/M$,  $M\in \mathds{N}$, $t_n=n\Delta t, n=0,1,\cdots,M$
be the mesh points.

Let $v(x,t)=\frac{\partial u(x,t)}{\partial t}$, and from the fact
$$v(x,t_i)=\frac{\partial u(x,t_i)}{\partial t}=\frac{3 u(x,t_i)-4 u(x,t_{i-1})+ u(x,t_{i-2})}{2\Delta t}+r^n_1,$$
where the truncation error $|r^n_1|\leq C (\Delta t)^2$, we can obtain
\begin{equation}\label{tf1}
\aligned \frac{\partial^{\alpha}u(x,t_n)}{\partial
t^{\alpha}}&=\frac{1}{\Gamma(2-\alpha)}\int_0^{t_n}\frac{\partial
v(x,s)}{\partial s}\frac{ds}{(t_n-s)^{\alpha-1}}\\
&=\frac{1}{\Gamma(2-\alpha)}\sum\limits_{i=0}^{n-1}\int_{t_i}^{t_{i+1}}\frac{\partial
v(x,s)}{\partial s}\frac{ds}{(t_n-s)^{\alpha-1}}\\
&=\frac{1}{\Gamma(2-\alpha)}\sum\limits_{i=0}^{n-1}\int_{t_i}^{t_{i+1}}\frac{v(x,t_{i+1})-v(x,t_{i})}{\Delta
t}\frac{ds}{(t_n-s)^{\alpha-1}}+r^n_2\\
&=\frac{(\Delta
t)^{2-\alpha}}{\Gamma(3-\alpha)}\sum\limits_{i=0}^{n-1}b_{n-i-1}\frac{v(x,t_{i+1})-v(x,t_{i})}{\Delta
t}+r^n_2\\
&=\frac{(\Delta
t)^{1-\alpha}}{\Gamma(3-\alpha)}[v(x,t_n)+\sum\limits_{i=1}^{n-1}(b_{n-i}-b_{n-i-1})v(x,t_i)-b_{n-1}v(x,t_0)]+r^n_2\\ &=\frac{(\Delta
t)^{1-\alpha}}{\Gamma(3-\alpha)}[\frac{3 u(x,t_n)-4 u(x,t_{n-1})+ u(x,t_{n-2})}{2\Delta t}\\
&~~~~~+\sum\limits_{i=1}^{n-1}(b_{n-i}-b_{n-i-1})\frac{3 u(x,t_i)-4 u(x,t_{i-1})+ u(x,t_{i-2})}{2\Delta t}\\
&~~~~~-b_{n-1}v(x,t_0)]+r^n_3, \endaligned
\end{equation}

where $$b_0=1,~~~~b_i=(i+1)^{2-\alpha}-i^{2-\alpha},i=1,2,3,\cdots$$ when $i=1$, we take
$u(x,-1)=u(x,0)-\Delta tu_1(x)+C(\Delta t)^2$ by Taylor expansion.

Similar to the proof in \cite{sp3},the truncation error $|r_2^n|\leq C (\Delta t)^{3-\alpha}$ , so
$r_3^n$ satisfied
$$|r_3^n|\leq C (\Delta t)^{3-\alpha}.$$

It is easy to check that
\begin{equation}\label{b}
\begin{split}
&b_i>0, i=1,2\cdots, n. \\
&1=b_0>b_1>b_2>\cdots >b_n, b_n\rightarrow 0(n\rightarrow\infty).
\end{split}
\end{equation}

Substituting (\ref{tf1}) into (\ref{question}), we have
\begin{equation*}
\aligned 3 u(x,t_n)-\beta\frac{\partial^{2}(x,t_n)}{\partial x^{2}}=&\sum\limits_{i=1}^{n-1}(b_{n-i-1}-b_{n-i})(3 (x,t_i)-4 u(x,t_{i-1})+ u(x,t_{i-2}))\\
&+2\Delta t b_{n-1}v(x,t_0)+\beta f(x,t_n)+4u(x,t_{n-1})\\
&-u(x,t_{n-2})+\beta r_3^n, \endaligned
\end{equation*}

where $\beta=2(\Delta t)^{\alpha} \Gamma(3-\alpha)$.

Let $u^k$ be the numerical approximation to $u(x,t_k)$, $f^n=f(x,t_n)$, the problem (\ref{question})
can be discretized by the following scheme
\begin{equation}\label{schemetime}
\aligned 3 u^n-\beta\frac{\partial^{2}u^n}{\partial x^{2}}=&\sum\limits_{i=1}^{n-1}(b_{n-i-1}-b_{n-i})(3 u^i-4 u^{i-1}+ u^{i-2})\\
&+2\Delta t b_{n-1}v^0+\beta f^n+4 u^{n-1}- u^{n-2}, \endaligned
\end{equation}
where $u^{-1}=u^0-\Delta tu_1(x).$ We know
$$|\beta r_3^n|\leq C(\Delta t)^3,$$
however, by Taylor expansion we have
$$|u(x,t_{-1})-u(x,0)+\Delta tu_1(x)|\leq C(\Delta t)^{2}, $$

therefore the truncation error is $O(\Delta t)^2$ in scheme (\ref{schemetime}).

\subsection{Fully discrete schemes}
In this subsection we present the fully discrete LDG scheme for the problem (\ref{question}) based on the semidiscrete scheme (\ref{schemetime}).

We rewrite Eq. (\ref{question}) as a first-order system:
\begin{eqnarray}\label{eq3}
&&p=u_x,~~~~~~\frac{\partial^{\alpha} u(x,t)}{\partial
t^{\alpha}}-p_x=f(x,t).
\end{eqnarray}

Let $u_h^n, p_h^n\in V_h^k$ be the approximations of $u(\cdot,t_n),
p(\cdot,t_n)$, respectively, $f^n(x)=f(x,t_n)$. We define a fully discrete local discontinuous
Galerkin scheme as follows: find $u_h^n, p_h^n\in V_h^k,$ such that
for all test functions $\phi,w\in V_h^k,$
\begin{equation}\label{scheme}
\begin{split}
3\int_{\Omega}u_h^n\phi dx&+\beta(\int_{\Omega}p_h^n\phi_xdx-\sum\limits_{j=1}^{N}((\widehat{p_h^n}\phi^-)_{j+\frac{1}{2}}-(\widehat{p_h^n}\phi^+)_{j-\frac{1}{2}}))\\
=&\sum\limits_{i=1}^{n-1}(b_{n-i-1}-b_{n-i})\int_{\Omega}(3u_h^{i}-4u_h^{i-1}+u_h^{i-2})\phi dx+2\Delta t b_{n-1}\int_{\Omega}v_h^{0}\phi dx\\
&+4\int_{\Omega}u_h^{n-1}\phi dx-\int_{\Omega}u_h^{n-2}\phi dx+\beta\int_{\Omega}f^n\phi dx,\\
\int_{\Omega}p_h^nwdx&+\int_{\Omega}u_h^nw_xdx-\sum\limits_{j=1}^{N}((\widehat{u_h^n}w^-)_{j+\frac{1}{2}}-(\widehat{u_h^n}w^+)_{j-\frac{1}{2}})=0,
\end{split}
\end{equation}
The initial conditions $u_h^{-1}, u_h^0, v_h^0$ are taken as the $L^2$ projections of
$u(¡¤,-1), u(¡¤,0), u_1(¡¤,0)$, respectively,
\begin{equation}
\begin{split}
\int_{\Omega}u_h^{-1}\phi dx&=\int_{\Omega}\mathcal{P}u(x,-1)\phi dx=\int_{\Omega}u(x,-1)\phi dx,\\
\int_{\Omega}u_h^{0}\phi dx&=\int_{\Omega}\mathcal{P}u(x,0)\phi dx=\int_{\Omega}u_0(x)\phi dx,\\
\int_{\Omega}v_h^{0}\phi dx&=\int_{\Omega}\mathcal{P}u_1(x,0)\phi dx=\int_{\Omega}u_1(x)\phi dx,
~~~\forall v\in V_h^k.
\end{split}
\end{equation}

The ``hat" terms in (\ref{scheme}) in the cell boundary terms from
integration by parts are the so-called ``numerical fluxes", which
are single valued functions defined on the edges and should be
designed based on different guiding principles for different PDEs to
ensure stability. It turns out that we can take the simple choices
such that
\begin{eqnarray}\label{flux1}
&&\widehat{u_h^n}=(u_h^n)^-,~~\widehat{p_h^n}=(p_h^n)^+.
\end{eqnarray}

We remark that the choice for the fluxes (\ref{flux1}) is not
unique. In fact the crucial part is taking $\widehat{u_h^n}$ and
$\widehat{p_h^n}$ from opposite sides \cite{xia,cs982}.

\subsection{Stability and Convergence}
In order to simplify the notations and without lose of generality,
we consider the case $f=0$ in its numerical analysis.

\begin{Theorem}\label{sta} For periodic or compactly supported boundary conditions,
the fully-discrete LDG scheme (\ref{scheme}) is unconditionally
stable, and there exists a positive constant $C$ depending on
$u,T,\alpha$, such that
\begin{equation}
\begin{split}
\|u_h^{n}\|\leq C(\|u_h^{0}\|+\Delta t\|u_1(x)\|), ~~~~~n=1,2\cdots,
M.
\end{split}\end{equation}
\end{Theorem}
\begin{proof} Taking $\phi=u^n_h, w=\beta p^n_h$ in
scheme (\ref{scheme}), we
obtain
\begin{equation*}
\begin{split}
3&\|u_h^n\|^2+\beta\|p_h^n\|^2+\beta\sum\limits_{j=1}^{N}(\Psi(u_h^n,p_h^n)_{j+\frac{1}{2}}
-\Psi(u_h^n,p_h^n)_{j-\frac{1}{2}}+\Theta(u_h^n,p_h^n)_{j-\frac{1}{2}})\\
\end{split}
\end{equation*}
\begin{equation}\label{sta2}
\begin{split}=&\sum\limits_{i=1}^{n-1}(b_{n-i-1}-b_{n-i})\int_{\Omega}(3u_h^{i}-4u_h^{i-1}+u_h^{i-2})u_h^ndx+2\Delta tb_{n-1}\int_{\Omega}v_h^{0}u_h^ndx\\
&+4\int_{\Omega}u_h^{n-1}u_h^ndx-\int_{\Omega}u_h^{n-2}u_h^ndx,\\
\end{split}
\end{equation}
where
\begin{equation*}
\begin{split}
\Psi(u_h^n,p_h^n)=&(p_h^n)^-(u_h^n)^--\widehat{p_h^n}(u_h^n)^--\widehat{u_h^n}(p_h^n)^-,\\
\Theta(u_h^n,p_h^n)
=&(p_h^n)^-(u_h^n)^--(p_h^n)^+(u_h^n)^+-\widehat{p_h^n}(u_h^n)^-+\widehat{p_h^n}(u_h^n)^+
-\widehat{u_h^n}(p_h^n)^-\\
&+\widehat{u_h^n}(p_h^n)^+.
\end{split}
\end{equation*}

If we take the fluxes (\ref{flux1}), after some manual calculation,
we can easily obtain $\Theta(u_h^n,p_h^n)=0.$

Then based on the equation (\ref{sta2}), we can get
\begin{equation*}
\begin{split}
3\|u_h^n\|^2+\beta\|p_h^n\|^2
=&\sum\limits_{i=1}^{n-1}(b_{n-i-1}-b_{n-i})\int_{\Omega}(3u_h^{i}-4u_h^{i-1}+u_h^{i-2})u_h^ndx\\
&+2\Delta tb_{n-1}\int_{\Omega}v_h^{0}u_h^ndx
+4\int_{\Omega}u_h^{n-1}u_h^ndx\\
&-\int_{\Omega}u_h^{n-2}u_h^ndx\\
\leq&\sum\limits_{i=1}^{n-1}(b_{n-i-1}-b_{n-i})(3\|u_h^{i}\|+4\|u_h^{i-1}\|+\|u_h^{i-2}\|)\|u_h^n\|\\
&+2\Delta tb_{n-1}\|v_h^{0}\|\|u_h^n\|
+4\|u_h^{n-1}\|\|u_h^n\|\\
&+\|u_h^{n-2}\|\|u_h^n\|,
\end{split}\end{equation*}
that is
\begin{equation}\label{staend}
\begin{split}
3\|u_h^n\|\leq&\sum\limits_{i=1}^{n-1}(b_{n-i-1}-b_{n-i})(3\|u_h^{i}\|+4\|u_h^{i-1}\|+\|u_h^{i-2}\|)\\
&+2\Delta tb_{n-1}\|v_h^{0}\|
+4\|u_h^{n-1}\|\\
&+\|u_h^{n-2}\|.
\end{split}\end{equation}

We will prove the Theorem \ref{sta}  by mathematical induction. When
$n=1$, we can obtain
\begin{equation}\label{stan=1}
\begin{split}
3\|u_h^1\|\leq 2\Delta t \|v_h^{0}\|
+4\|u_h^{0}\|
+\|u_h^{-1}\|\\
\end{split}
\end{equation}

Notice that
$$
\int_{I_j}u_h^{-1}vdx=\int_{I_j}\mathbb{P}(u(x,0)-\Delta t
u_1(x))vdx=\int_{I_j}u_h^0vdx-\Delta t \int_{I_j}u_1(x)vdx,
$$
for any $v\in V_h^k$. Taking $v=u_h^{-1}$, we can obtain
\begin{equation*}
\begin{split}
\|u_h^{-1}\|_{I_j}^2&=\int_{I_j}u_h^0u_h^{-1}dx-\Delta t
\int_{I_j}u_1(x)u_h^{-1}dx\\
&\leq \|u_h^{0}\|_{I_j}^2+\frac{1}{4}\|u_h^{-1}\|_{I_j}^2+(\Delta
t)^2\|u_1(x)\|_{I_j}^2+\frac{1}{4}\|u_h^{-1}\|_{I_j}^2,
\end{split}\end{equation*}
summing over $j$ from $1$ to $N$, we can get
\begin{equation}\label{sta-1}
\begin{split}
\|u_h^{-1}\| &\leq C(\|u_h^{0}\|+\Delta t\|u_1(x)\|).
\end{split}\end{equation}

Similar to the proof of (\ref{sta-1}), we can easily obtain
\begin{equation}\label{sta-2}
\|v_h^{0}\| \leq \|u_1(x)\|.
\end{equation}
By using (\ref{stan=1}),(\ref{sta-1}) and (\ref{sta-2}), it is easily to know that there exists a
positive constant $C$, such that
\begin{equation}\label{lastn=1}
\begin{split}
\|u_h^{1}\|&\leq C(\|u_h^{0}\|+\Delta t\|u_1(x)\|).
\end{split}\end{equation}

Now suppose the following inequality holds
\begin{equation}\label{assump1}
\begin{split}
\|u_h^m\|\leq C(\|u_h^{0}\|+\Delta t\|u_1(x)\|), m=2,3\cdots K,
\end{split}\end{equation}
we need to prove $\|u_h^{K+1}\|\leq C(\|u_h^{0}\|+\Delta
t\|u_1(x)\|).$

Let $n=K+1$ in the inequality (\ref{staend}), we can
obtain
\begin{equation*}
\begin{split}
3\|u_h^{K+1}\|\leq&\sum\limits_{i=1}^{K}(b_{K-i}-b_{K+1-i})(3\|u_h^{i}\|+4\|u_h^{i-1}\|+\|u_h^{i-2}\|)\\
&+2\Delta tb_{K}\|v_h^{0}\|
+4\|u_h^{K}\|\\
&+\|u_h^{K-1}\|.\end{split}
\end{equation*}

Using (\ref{sta-1}), (\ref{sta-2}) and (\ref{assump1}),
we can obtain the following inequality easily
$$\|u_h^{K+1}\|\leq C(\|u_h^{0}\|+\Delta t\|u_1(x)\|).$$

This finishes the proof of the stability result.
\end{proof}
\begin{Theorem}\label{errorresult} Let $u(x,t_n)$ be the exact solution of
problem (\ref{question}), which is sufficiently smooth such that
$u\in H^{m+1}$ with $0\leq m\leq k+1$. Let $u_h^n$ be the numerical
solution of the fully discrete LDG scheme (\ref{scheme}), then there
holds the following error estimate:
\begin{equation}\label{err}
\begin{split}
\|u(x,t_n)-u_h^n\|\leq C(h^{k+1}+(\Delta t)^{2}), n=1,\cdots,M,
\end{split}
\end{equation}
where C is a constant depending on $u, T, \alpha$.
\end{Theorem}
\begin{proof} By Taylor expansion we know
$$|u(x,t_{-1})-u(x,0)+\Delta tu_1(x)|\leq C(\Delta t)^{2}, $$
here $C$ is a positive constant depending on $u$. Then by using the
property (\ref{projection111}), we can obtain the following estimate
which will be used later,
\begin{equation}
\|u(x,t_{-1})-u_h^{-1}\|\leq C((\Delta t)^{2}+h^{k+1}).
\end{equation}

It is easy to verify that the exact solution of PDE (\ref{question})
satisfies
\begin{equation}\label{weak}
\begin{split}
3 \int_{\Omega}u(x,t_{n})\phi dx&+\beta (\int_{\Omega}p(x,t_{n})\phi_xdx-\sum\limits_{j=1}^{N}((p(x,t_{n})\phi^-)_{j+\frac{1}{2}}-(p(x,t_{n})\phi^+)_{j-\frac{1}{2}}))\\
=&\sum\limits_{i=1}^{n-1}(b_{n-i-1}-b_{n-i})\int_{\Omega}(3u(x,t_{i})-4u(x,t_{i-1})+u(x,t_{i-2}))\phi dx\\
&+2\Delta t  b_{n-1}\int_{\Omega}v(x,t_{0})\phi dx+4 \int_{\Omega}u(x,t_{n-1})\phi dx\\
&- \int_{\Omega}u(x,t_{n-2})\phi dx+\beta\int_{\Omega}f(x,t_{n})\phi dx+\beta\int_{\Omega}r_3^n\phi dx,\\
\int_{\Omega}p(x,t_{n})wdx&+\int_{\Omega}u(x,t_{n})w_xdx-\sum\limits_{j=1}^{N}((u(x,t_{n})w^-)_{j+\frac{1}{2}}-(u(x,t_{n})w^+)_{j-\frac{1}{2}})=0,
\end{split}
\end{equation}
$\forall v,\eta\in H^1(I_j)$, for $j=1,\cdots N$.

Denote
\begin{equation}\label{errnotation}
\begin{split}
&e_u^n=u(x,t_n)-u_h^n=\mathcal{P^-}e_u^n-(\mathcal{P^-}u(x,t_n)-u(x,t_n)),\\
&e_p^n=p(x,t_n)-p_h^n=\mathcal{P^+}e_p^n-(\mathcal{P^+}p(x,t_n)-p(x,t_n)).
\end{split}
\end{equation}

Subtracting (\ref{scheme}) from (\ref{weak}), and with the fluxes
(\ref{flux1}) we can obtain the error equation:
\begin{equation}\label{err1}
\begin{split}
3 \int_{\Omega}e_u^n\phi dx&+\beta(\int_{\Omega}e_p^n\phi_xdx-\sum\limits_{j=1}^{N}(((e_p^n)^+\phi^-)_{j+\frac{1}{2}}-((e_p^n)^+\phi^+)_{j-\frac{1}{2}}))\\
-&\sum\limits_{i=1}^{n-1}(b_{n-i-1}-b_{n-i})\int_{\Omega}(3e_u^{i}-4e_u^{i-1}+e_u^{i-2})\phi dx\\
&-4 \int_{\Omega}e_u^{n-1}\phi dx\\
&+ \int_{\Omega}e_u^{n-2}\phi dx+\beta\int_{\Omega}r_3^n\phi dx+\int_{\Omega}e_p^nwdx+\int_{\Omega}e_u^nw_xdx\\
&-\sum\limits_{j=1}^{N}(((e_u^n)^-w^-)_{j+\frac{1}{2}}-((e_u^n)^-w^+)_{j-\frac{1}{2}})
=0.
\end{split}
\end{equation}

Using (\ref{errnotation}), the error equation (\ref{err1}) can be
written as follows:
\begin{equation*}
\begin{split}
3 \int_{\Omega}&\mathcal{P^-}e_u^n\phi dx+\beta(\int_{\Omega}\mathcal{P^+}e_p^n\phi_xdx-\sum\limits_{j=1}^{N}(((\mathcal{P^+}e_p^n)^+\phi^-)_{j+\frac{1}{2}}-((\mathcal{P^+}e_p^n)^+\phi^+)_{j-\frac{1}{2}}))\\
&+\int_{\Omega}\mathcal{P^+}e_p^nwdx+\int_{\Omega}\mathcal{P^-}e_u^nw_xdx
-\sum\limits_{j=1}^{N}(((\mathcal{P^-}e_u^n)^-w^-)_{j+\frac{1}{2}}-((\mathcal{P^-}e_u^n)^-w^+)_{j-\frac{1}{2}})\\
=&\sum\limits_{i=1}^{n-1}(b_{n-i-1}-b_{n-i})\int_{\Omega}(3\mathcal{P^-}e_u^{i}-4\mathcal{P^-}e_u^{i-1}+\mathcal{P^-}e_u^{i-2})\phi dx\\
&+4 \int_{\Omega}\mathcal{P^-}e_u^{n-1}\phi dx
- \int_{\Omega}\mathcal{P^-}e_u^{n-2}\phi dx-\beta\int_{\Omega}r_3^n\phi dx\\
\end{split}
\end{equation*}
\begin{equation}\label{err2}
\begin{split}
&-\sum\limits_{i=1}^{n-1}(b_{n-i-1}-b_{n-i})\int_{\Omega}(3(\mathcal{P^-}u(x,t_i)-u(x,t_i))-4(\mathcal{P^-}u(x,t_{i-1})-u(x,t_{i-1}))\\
&+(\mathcal{P^-}u(x,t_{i-2})-u(x,t_{i-2})))\phi dx-4 \int_{\Omega}(\mathcal{P^-}u(x,t_{n-1})-u(x,t_{n-1}))\phi dx\\
&+ \int_{\Omega}(\mathcal{P^-}u(x,t_{n-2})-u(x,t_{n-2}))\phi dx\\
&+3 \int_{\Omega}(\mathcal{P^-}u(x,t_n)-u(x,t_n))\phi dx+\beta(\int_{\Omega}(\mathcal{P^+}p(x,t_n)-p(x,t_n))\phi_xdx\\
&-\sum\limits_{j=1}^{N}((((\mathcal{P^+}p(x,t_n)-p(x,t_n)))^+\phi^-)_{j+\frac{1}{2}}-(((\mathcal{P^+}p(x,t_n)-p(x,t_n)))^+\phi^+)_{j-\frac{1}{2}}))\\
&+\int_{\Omega}(\mathcal{P^+}p(x,t_n)-p(x,t_n))wdx+\int_{\Omega}(\mathcal{P^-}u(x,t_n)-u(x,t_n))w_xdx\\
&-\sum\limits_{j=1}^{N}(((\mathcal{P^-}u(x,t_n)-u(x,t_n))^-w^-)_{j+\frac{1}{2}}-((\mathcal{P^-}u(x,t_n)-u(x,t_n))^-w^+)_{j-\frac{1}{2}}).
\end{split}
\end{equation}

Taking the test functions $\phi=\mathcal{P^-}e_u^n, w=\beta
\mathcal{P^+}e_p^n$ in (\ref{err2}), using the properties
(\ref{pro1})-(\ref{pro2}), then the following equality holds,
\begin{equation*}
\begin{split}
3\int_{\Omega}&(\mathcal{P^-}e_u^n)^2dx+\beta\int_{\Omega}(\mathcal{P^+}e_p^n)^2dx\\
=&\sum\limits_{i=1}^{n-1}(b_{n-i-1}-b_{n-i})\int_{\Omega}(3\mathcal{P^-}e_u^{i}-4\mathcal{P^-}e_u^{i-1}+\mathcal{P^-}e_u^{i-2})\mathcal{P^-}e_u^n dx\\
&+4\int_{\Omega}\mathcal{P^-}e_u^{n-1}\mathcal{P^-}e_u^n dx
-\int_{\Omega}\mathcal{P^-}e_u^{n-2}\mathcal{P^-}e_u^n dx\\
&-\sum\limits_{i=1}^{n-1}(b_{n-i-1}-b_{n-i})\int_{\Omega}(3(\mathcal{P^-}u(x,t_i)-u(x,t_i))-4(\mathcal{P^-}u(x,t_{i-1})\\
&-u(x,t_{i-1}))+(\mathcal{P^-}u(x,t_{i-2})-u(x,t_{i-2})))\mathcal{P^-}e_u^n dx-4\\
&\int_{\Omega}(\mathcal{P^-}u(x,t_{n-1})-u(x,t_{n-1}))\mathcal{P^-}e_u^n dx-\beta\int_{\Omega}r_3^n\mathcal{P^-}e_u^n dx\\
&+\int_{\Omega}(\mathcal{P^-}u(x,t_{n-2})-u(x,t_{n-2}))\mathcal{P^-}e_u^n dx\\
&+3\int_{\Omega}(\mathcal{P^-}u(x,t_n)-u(x,t_n))\mathcal{P^-}e_u^n dx\\
&+\beta\int_{\Omega}(\mathcal{P^+}p(x,t_n)-p(x,t_n))\mathcal{P^+}e_p^ndx.
\end{split}
\end{equation*}

Therefore, we obtain
\begin{equation*}
\begin{split}
\|\mathcal{P^-}e_u^n)\|
\leq&\sum\limits_{i=1}^{n-1}(b_{n-i-1}-b_{n-i})(3\|\mathcal{P^-}e_u^{i}\|+4\|\mathcal{P^-}e_u^{i-1}\|+\|\mathcal{P^-}e_u^{i-2}\|) \\
&+4\|\mathcal{P^-}e_u^{n-1}\|
+\|\mathcal{P^-}e_u^{n-2}\|+\beta\|r_3^n\|\\
&+\sum\limits_{i=1}^{n-1}(b_{n-i-1}-b_{n-i})(3\|\mathcal{P^-}u(x,t_i)-u(x,t_i)\|\\
&+4\|\mathcal{P^-}u(x,t_{i-1})-u(x,t_{i-1})\|+\|\mathcal{P^-}u(x,t_{i-2})-u(x,t_{i-2})\|)\\
\end{split}
\end{equation*}
\begin{equation}\label{err3}
\begin{split}
&+ \|\mathcal{P^-}u(x,t_{n-2})-u(x,t_{n-2})\|
+4 \|\mathcal{P^-}u(x,t_{n-1})-u(x,t_{n-1})\|\\
&+3 \|\mathcal{P^-}u(x,t_n)-u(x,t_n)\|+\sqrt{\beta}\|\mathcal{P^+}p(x,t_n)-p(x,t_n)\|.
\end{split}
\end{equation}

We prove the error estimates (\ref{err}) by mathematical induction.
First, we consider the case when $n=1$. From (\ref{err1}) and (\ref{err3}) , we know
\begin{equation}\label{err3}
\begin{split}
\|\mathcal{P^-}e_u^1)\|
\leq&4\|\mathcal{P^-}e_u^{0}\|
+\|e_u^{-1}\|+\beta\|r_3^1\|\\
&+4\|\mathcal{P^-}u(x,t_{0})-u(x,t_{0})\|\\
&+3\|\mathcal{P^-}u(x,t_1)-u(x,t_1)\|+\sqrt{\beta}\|\mathcal{P^+}p(x,t_1)-p(x,t_1)\|
\end{split}
\end{equation}

Notice the facts that $$\mathcal{P^-}e_u^{0}=0,~~~\|e_u^{-1}\|\leq
C(h^{k+1}+(\Delta t)^2),~~\|r_3^1\|\leq C(\Delta t)^{3-\alpha},$$ and the
property (\ref{projection111}), we can obtain
\begin{equation}\label{err4}
\begin{split}
&\|\mathcal{P^-}e_u^1\|\leq C(h^{k+1}+(\Delta t)^{2}).
\end{split}
\end{equation}

Next we suppose the following inequality holds
\begin{equation}\label{assump}\|\mathcal{P^-}e_u^m\|\leq
C(h^{k+1}+(\Delta t)^{2}),
m=1,2,\cdots K.
\end{equation}

When $n=K+1$, from the equation (\ref{err3}), we can obtain
\begin{equation*}
\begin{split}
\|\mathcal{P^-}e_u^{K+1})\|
\leq&\sum\limits_{i=1}^{K}(b_{K-i}-b_{K+1-i})(3\|\mathcal{P^-}e_u^{i}\|+4\|\mathcal{P^-}e_u^{i-1}\|+\|\mathcal{P^-}e_u^{i-2}\|) \\
&+4 \|\mathcal{P^-}e_u^{K}\|
+ \|\mathcal{P^-}e_u^{K-1}\|+\beta\|r_3^{K+1}\|\\
&+\sum\limits_{i=1}^{K}(b_{K-i}-b_{K+1-i})(3\|\mathcal{P^-}u(x,t_i)-u(x,t_i)\|+4\|\mathcal{P^-}u(x,t_{i-1})-u(x,t_{i-1})\|\\
&+\|\mathcal{P^-}u(x,t_{i-2})-u(x,t_{i-2})\|)+4 \|\mathcal{P^-}u(x,t_{K})-u(x,t_{K})\|\\
&+ \|\mathcal{P^-}u(x,t_{K-1})-u(x,t_{K-1})\|\\
&+3 \|\mathcal{P^-}u(x,t_{K+1})-u(x,t_{K+1})\|+\sqrt{\beta}\|\mathcal{P^+}p(x,t_{K+1})-p(x,t_{K+1})\|\\
\leq&\sum\limits_{i=1}^{K}(b_{K-i}-b_{K+1-i})C(h^{k+1}+(\Delta t)^{2}) \\
&+4 C(h^{k+1}+(\Delta t)^{2})
+ C(h^{k+1}+(\Delta t)^{2})+C(\Delta t)^{3}\\
&+\sum\limits_{i=1}^{K}(b_{K-i}-b_{K+1-i})Ch^{k+1}+Ch^{k+1}+C(\Delta t)^{\frac{\alpha}{2}}h^{k+1}.
\end{split}
\end{equation*}

Similar to the proof of (\ref{err4}), we can obtain the following
result immediately
$$\|\mathcal{P^-}e_u^{K+1}\|\leq C(h^{k+1}+(\Delta t)^{2}).$$

Thus Theorem \ref{errorresult} follows by the triangle inequality
and the interpolation property (\ref{projection111}).
\end{proof}
\section{Numerical examples}
In this section, we present numerical experiments of the presented finite difference/local discontinuous Galerkin method to the fractional diffusion-wave equation to verify the error estimates in Section 3.

\textbf{Example 4.1.}  Consider the following fractional diffusion-wave equation
\begin{equation} \label{ex1}
\begin{split}
&\frac{\partial^{\alpha} u(x,t)}{\partial
t^{\alpha}}-\frac{\partial^{2}u(x,t)}{\partial x^{2}}=f(x,t),
~~~~~~(x,t)\in [0,1]\times [0,1],\\
&u(x,0)=0,~~~~\frac{\partial u(x,0)}{\partial
t}=0,~~x\in [0,1]. \\
\end{split}
\end{equation}

Choose a suitable right-hand-side function $f(x,t)$ such that the exact solution is
$$u(x,t)=t^2\sin (2\pi x)$$.
Table 1-4 display the errors in $L^2$-norm and
$L^\infty$-norm at $T=1$ and convergence orders in space for piecewise $P^k$ polynomials for several values of $\alpha:1.2,1.4, 1.6, 1.8$, with time step $\Delta t=1/1000$. Obviously the $(k+1)$-th order of accuracy in space are observed, which is in agreement with the theoretic results.

In order to investigate the temporal accuracy of the proposed method, we fix the space step $h=1/200$. Table 5 show that the errors in $L^2$-norm and
$L^1$-norm attain the second-order convergence in time. The results are consistent with our theoretical results in Theorem \ref{errorresult}.

In Figure \ref{figure} we plot the the approximate solution of the three order on the uniform mesh with $100$ cells and the exact solution at $T=1$ to show the performance of the presented scheme. We can see that the method is very effective and is a good tool to solve such problems.

\begin{table}[htpb]
\caption{Spatial accuracy test for the time-fractional diffusion-wave equation
(\ref{ex1}) using piecewise $P^k$ polynomials. $\alpha=1.2, \Delta t=\frac{1}{1000}, T=1$. } \centering
\begin{tabular}{|c|c|c|c|c|c|c|c|}
 \hline
   & N& $L^2$-error & order & $L^\infty$-error & order\\
\hline
& 5&0.265109983989909 & - &0.623532065154133& -\\
& 10&0.129308265170869& 1.03 &0.313595441923762 &0.98 \\
$P^0$ & 20&6.425848539792525E-002& 1.01  &0.157010259108059 &1.00 \\
& 40&3.208010396964848E-002&1.00 &7.853125916762804E-002 & 1.00 \\
& 80&1.603391954016398E-002&1.00 &3.926891940062052E-002& 1.00\\
  \hline
  & 5&6.736979744152280E-002& - &0.249880240379731& -\\
& 10&1.695495641564284E-002&1.99 &6.468476942047330E-002&1.95 \\
$P^1$ & 20&4.245128618901225E-003& 2.00 &1.631233665625631E-002 &1.99  \\
& 40&1.061671545606701E-003&2.00&4.103801542149732E-003 & 1.99\\
& 80&2.654420510184321E-004&2.00&1.027598756713433E-003 & 2.00\\
  \hline
  & 5&6.682959934132981E-003 & - &3.174066978350254E-002& -\\
& 10&8.506996364720942E-004&2.97   &3.971254358826398E-003& 3.00\\
$P^2$ & 20&1.068204883018595E-004&2.99 & 5.116352220441455E-004& 2.96  \\
& 40&1.336779544959365E-005& 3.00&6.443554411463790E-005&2.99 \\
& 80&1.672333408099435E-006& 3.00&8.069512789853388E-006&3.00 \\
  \hline
\end{tabular}
\end{table}

\begin{table}[htpb]
\caption{Spatial accuracy test forthe time-fractional diffusion-wave equation
(\ref{ex1}) using piecewise $P^k$ polynomials. $\alpha=1.4, \Delta t=\frac{1}{1000}, T=1$. } \centering
\begin{tabular}{|c|c|c|c|c|c|c|c|}
 \hline
   & N& $L^2$-error & order & $L^\infty$-error & order\\
\hline
& 5&0.265002133818103& - &0.623263433950011& -\\
& 10&0.129296240434182& 1.03 &0.313564577675057&0.98 \\
$P^0$ & 20&6.425702562008787E-002& 1.01  &0.157006484631294 &1.00 \\
& 40&3.207992382929516E-002&1.00 &7.853079254830535E-002& 1.00 \\
& 80&1.603389759999361E-002&1.00 &3.926886254925913E-002& 1.00\\
  \hline
  & 5&6.736429618528465E-002& - &0.249848519443165& -\\
& 10&1.695467033281010E-002&1.99 & 6.468306250704692E-002&1.95 \\
$P^1$ & 20&4.245111564609582E-003& 2.00 &1.631223051257358E-002 &1.99  \\
& 40&1.061670492069727E-003&2.00&4.103794461966181E-003& 1.99\\
& 80&2.654419852360282E-004&2.00&1.027597816101400E-003 & 2.00\\
  \hline
  & 5&6.682553306760190E-003 & - &3.173870305522863E-002& -\\
& 10&8.506873719172307E-004&2.97   &3.971185742546351E-003& 3.00\\
$P^2$ & 20&1.068201082305646E-004&2.99 &5.116330685330400E-004& 2.96  \\
& 40&1.336778230870020E-005& 3.00&6.443547673697195E-005&2.99 \\
& 80&1.672322693664752E-006& 3.00& 8.069510669296185E-006&3.00 \\
  \hline
\end{tabular}
\end{table}

\begin{table}[htpb]
\caption{Spatial accuracy test forthe time-fractional diffusion-wave equation
(\ref{ex1}) using piecewise $P^k$ polynomials. $\alpha=1.6, \Delta t=\frac{1}{1000}, T=1$. } \centering
\begin{tabular}{|c|c|c|c|c|c|c|c|}
 \hline
   & N& $L^2$-error & order & $L^\infty$-error & order\\
\hline
& 5&0.264972688983567& - &0.623189941325662& -\\
& 10&0.129291634839670& 1.03 &0.313552730827037&0.98 \\
$P^0$ & 20&6.425644681449535E-002& 1.01  &0.157004984816329  &1.00 \\
& 40&3.207985524403432E-002&1.00 &7.853061451680375E-002& 1.00 \\
& 80&1.603389101024813E-002&1.00 &3.926884544112283E-002& 1.00\\
  \hline
  & 5&6.736317907360344E-002& - &0.249838735643862& -\\
& 10&1.695461845352325E-002&1.99 & 6.468254410369267E-002&1.95 \\
$P^1$ & 20&4.245108546494545E-003& 2.00 &1.631218386012478E-002&1.99  \\
& 40&1.061670300965039E-003&2.00&4.103777047141932E-003& 1.99\\
& 80&2.654419703069185E-004&2.00&1.027582241539760E-003& 2.00\\
  \hline
  & 5&6.682478515070153E-003 & - &3.173831828892504E-002& -\\
& 10&8.506852141270998E-004&2.97   &3.971173620498153E-003& 3.00\\
$P^2$ & 20&1.068200376865845E-004&2.99 &5.116326863045966E-004& 2.96  \\
& 40&1.336774817234452E-005& 3.00&6.443546414282400E-005&2.99 \\
& 80&1.672066664180233E-006& 3.00&8.069510170348080E-006&3.00 \\
  \hline
\end{tabular}
\end{table}

\begin{table}[htpb]
\caption{Spatial accuracy test for the time-fractional diffusion-wave equation
(\ref{ex1}) using piecewise $P^k$ polynomials, $\alpha=1.8, \Delta t=\frac{1}{1000}, T=1$. } \centering
\begin{tabular}{|c|c|c|c|c|c|c|c|}
 \hline
   & N& $L^2$-error & order & $L^\infty$-error & order\\
\hline
& 5&0.265395156138938& - &0.624238212597400& -\\
& 10&0.129328509614864& 1.03 &0.313647188207608&0.98 \\
$P^0$ & 20&6.426055758264328E-002& 1.01  &0.157015597179358 &1.00 \\
& 40&3.208037034587025E-002&1.00 &7.853194655385835E-002& 1.00 \\
& 80&1.603396389418712E-002&1.00 &3.926903380844882E-002& 1.00\\
  \hline
  & 5&6.736868534645306E-002& - &0.249898961167015& -\\
& 10&1.695486972560303E-002&1.99 & 6.468562726677418E-002&1.95 \\
$P^1$ & 20&4.245123063266223E-003& 2.00 &1.631242414011957E-002&1.99  \\
& 40&1.061671210932901E-003&2.00&4.103841431837951E-003& 1.99\\
& 80&2.654420391246725E-004&2.00& 1.027635651727810E-003& 2.00\\
  \hline
  & 5&6.682807553983295E-003 & - &3.174004623627680E-002& -\\
& 10&8.506954178648776E-004&2.97   &3.971230808741821E-003& 3.00\\
$P^2$ & 20&1.068203742828856E-004&2.99 &5.116345075529710E-004& 2.96  \\
& 40&1.336790404833714E-005& 3.00&6.443552350912754E-005&2.99 \\
& 80&1.673232029466757E-006& 3.00&8.070278651961527E-006&3.00 \\
  \hline
\end{tabular}
\end{table}

\begin{table}[htpb]
\caption{Temporal accuracy test for the time-fractional diffusion-wave equation
(\ref{ex1}) using piecewise $P^2$ polynomials. $N=200$} \centering
\begin{tabular}{|c|c|c|c|c|c|c|c|}
 \hline
   & $\Delta t$ & $L^2$-error & order & $L^1$-error & order\\
\hline
& 0.05&2.315895458864733E-006&-&2.078028449959208E-006 & - \\
& 0.04&1.452896981500066E-006 & 2.09 &1.301714631618507E-006& 2.10\\
$\alpha=1.1$ & 0.03&8.051939547113251E-007& 2.05 &7.198086134684791E-007&2.06\\
& 0.02&3.394311162501219E-007& 2.13  &3.023980936139506E-007 &2.14 \\
  \hline
  & 0.05&1.550654829179151E-004&-&1.396164504933172E-004& - \\
& 0.04&9.400505980906901E-005 & 2.24 &8.464278635512412E-005& 2.24\\
$\alpha=1.8$ & 0.03&5.271466986243190E-005& 2.01&4.746819651138549E-005&2.01\\
& 0.02&2.314548304581403E-005& 2.03  &2.101160952907237E-005&2.01 \\
  \hline
\end{tabular}
\end{table}

\begin{figure}[htpb]
\centering
\includegraphics[width=0.6\textwidth]{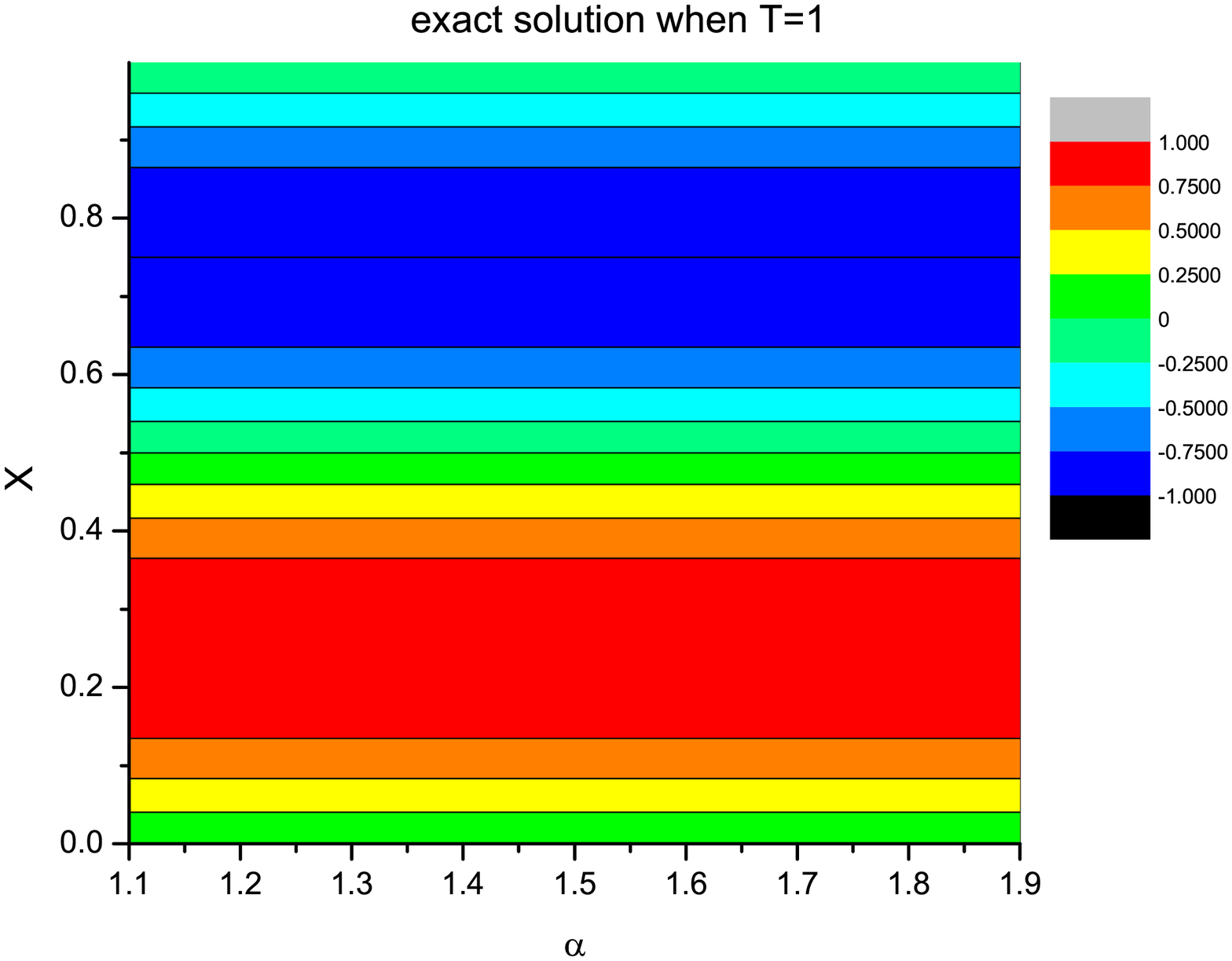}
\includegraphics[width=0.6\textwidth]{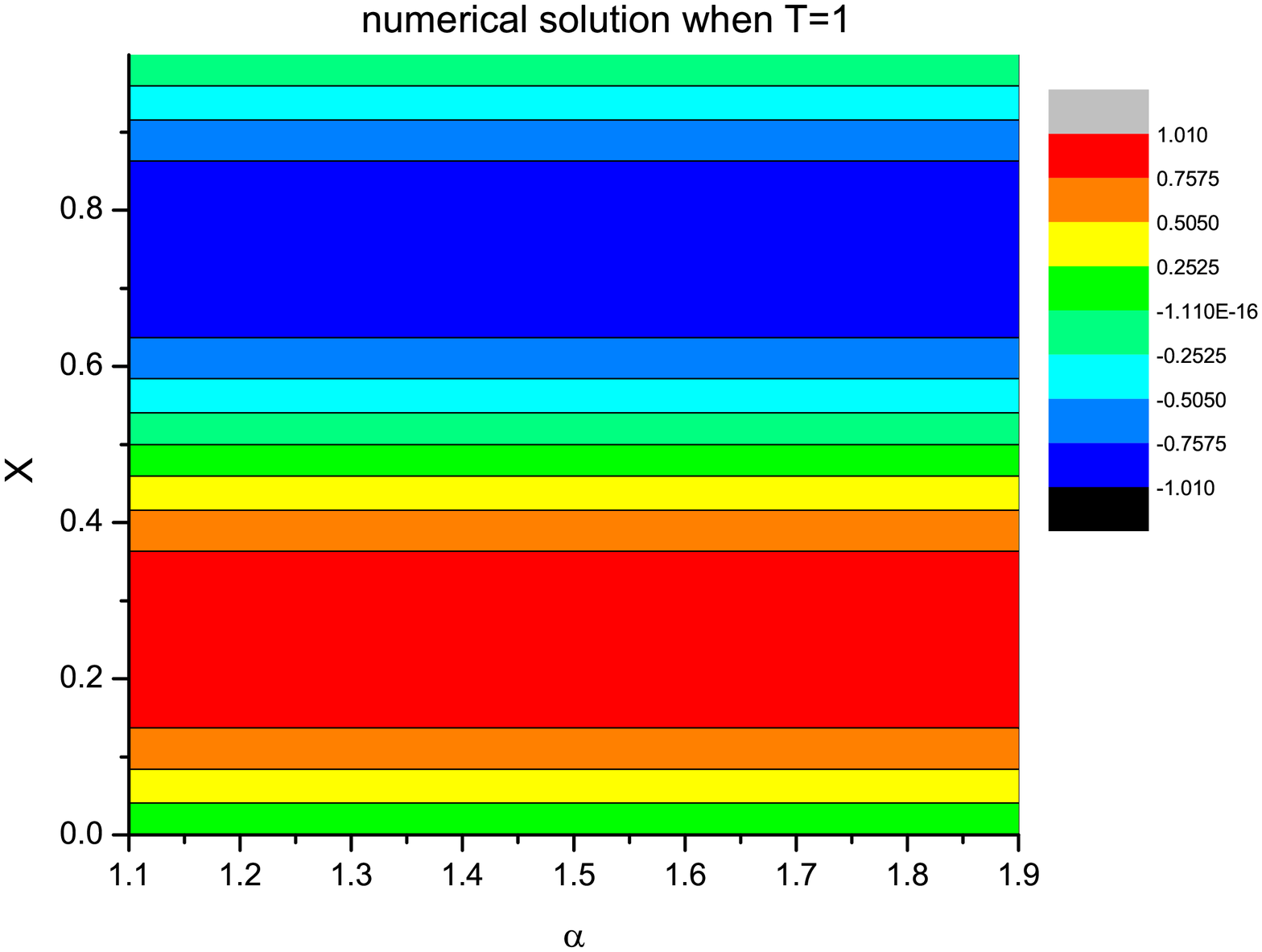}
\caption{The exact solution is contrasted against approximate
solution obtained on the uniform mesh with $100$ cells using $P^2$ elements
for Eq.(\ref{question}) when
T=1.}\label{figure}
\end{figure}

\section{Conclusion}

In this work, we have presented a finite difference/local discontinuous
Galerkin method for the fractional diffusion-wave equation. We first propose a finite difference method to approximate the time fractional derivatives when $1<\alpha<2$, and then give a fully discrete scheme and prove that the scheme is unconditionally stable and convergent. The extensive numerical example fully confirm the theoretic analysis.

\section*{Acknowledgement}
This work is supported by the High-Level Personal Foundation of Henan University of Technology (2013BS041), Plan For Scientific Innovation Talent of
Henan University of Technology (2013CXRC12), and the National Natural Science Foundation of China, Tian Yuan Special Foundation (11426090),and China Postdoctoral Science Foundation funded project (2015M572115).


\begin{thebibliography}{99}


\bibitem{op07} O. P. Agrawal, Analytical schemes for a new class of fractional differential equations, J. Phys. A: Math. Theor. 40 (2007), 5469.

\bibitem{fd1}  T. S. Basu and H. Wang, A fast second-order finite difference method for space-fractional
diffusion equations, Int. J. Numer. Anal. Modeling, 9 (2012), pp. 658-666.

\bibitem{sp1} A. R. Carella and C. A. Dorao, Least-squares spectral method for the solution of a fractional
advection-dispersion equation, J. Comput. Phys., 232 (2013) pp. 33-45.

\bibitem{chen08} C. Chen, F. Lui, K. Burrage, Finite difference methods
and a fourier analysis for the fractional reaction-subdiffusion
equation. Appl. Math. Comput. 198 (2008), pp. 754-769.

\bibitem{cs982} B. Cockburn, C.-W. Shu, The local discontinuous Galerkin method for time-dependent convection-diffusion systems, SIAM J. Numer. Anal. 35 (1998), pp. 2440-2463.

\bibitem{fd2} M. R. Cui, Compact finite difference method for the fractional diffusion equation, J. Comput.
Phys., 228 (2009), pp. 7792-7804.

\bibitem{ca97}  A. Compte, R.Metzler. The generalized Cattaneo equation for the description of anomalous transport
processes. J. Phys. A: Math. Gen. 30 (1997), pp. 7277-7289.

\bibitem{fe5}  W. Deng. Finite element method for the space and time fractional Fokker-Planck
equation. SIAM J. Numer. Anal. 47 (2008), pp. 204-226.

\bibitem{fd3} H. F. Ding and C. P. Li, Mixed spline function method for reaction-subdiffusion equation,
J. Comput. Phys., 242 (2013), pp. 103-123.

\bibitem{DC10} R. Du, W.R. Cao, Z.Z. Sun, A compact difference scheme for the
fractional diffusion-wave equation, Appl. Math. Model., 34 (2010),
pp. 2998-3007.

\bibitem{fe1}  V. J. Ervin, N. Heuer, and J. P. Roop, Numerical approximation of a time dependent,
nonlinear, space-fractional diffusion equation, SIAM J. Numer. Anal., 45 (2007), pp. 572-
591.


\bibitem{fe6} G. Fix, J. Roop, Least squares finite element solution of
a fractional order two-point boundary value problem, Comput. Math.
Appl. 48 (2004), pp. 1017-1033.

\bibitem{fd4} G. H. Gao and Z. Z. Sun, A compact finite difference scheme for the fractional sub-diffusion
equations, J. Comput. Phys., 230 (2011), pp. 586-595.

\bibitem{ot5} J.H. He and X.H. Wu. Variational iteration method: New development and applications,  Comput. Math. Appl., 54 (2007),pp. 881-894.

\bibitem{rh00} R. Hilfer, Ed., Applications of Fractional Calculus in Physics, World Scientific, Singapore, 2000.

\bibitem{fe7}Y. Jiang, J. Ma, High-order finite element methods for time-fractional partial
differential equations, J. Comput. Appl. Math. 235 (2011), pp.
3285-3290.

\bibitem{fe4} B. Jin, R.Lazarov, Y. Liu, Z.Zhou. The Galerkin finite element method for a multi-term
time-fractional diffusion equation. J. Comput. Phys., 281 (2015), pp. 825-843.


\bibitem{ks06} A. A. Kilbas, H. M. Srivastava, and J. J. Trujillo, Theory and Applications of Fractional Differential Equations, vol. 204, Elsevier, Amsterdam, The Netherlands, 2006.

\bibitem{rg08} R. Klages, G. Radons, and I. M. Sokolov, Eds., Anomalous Transport: Foundations and
Applications,Elsevier, Amsterdam, The Netherlands, 2008.

\bibitem{kti09} T.Kosztolowicz, K.D. Lewandowska. Hyperbolic subdiffusive impedance. J. Phys. A:Math. Theor. 42 (2009),
055004.

\bibitem{lk08}  K.D.Lewandowskaw. Application of generalized Cattaneo equation to model subdiffusion impedance.
Acta. Phys. Polonica. B. 39 (2008),  pp. 1211-1220.

\bibitem{fd5} T. A. M. Langlands and B. I. Henry, The accuracy and stability of an implicit solution
method for the fractional diffusion equation, J. Comput. Phys., 205 (2005), pp. 719-736.

\bibitem{sp2}X. J. Li and C. J. Xu, A space-time spectral method for the time fractional diffusion equation,
SIAM J. Numer. Anal., 47 (2009), pp. 2108-2131.

\bibitem{sp3}Y. M. Lin and C. J. Xu, Finite difference/spectral approximations for the time-fractional
diffusion equation, J. Comput. Phys., 225 (2007), pp. 1533-1552.

\bibitem{fd6} C. P. Li and F. H. Zeng, The finite difference methods for fractional ordinary differential
equations, Numer. Funct. Anal. Optim., 34 (2013), pp. 149-179.

\bibitem{fd7} F. Liu, P. Zhuang, and K. Burrage, Numerical methods and analysis for a class of fractional
advection-dispersion models, Comput. Math. Appl., 64 (2012), pp. 2990-3007.

\bibitem{fd8} M. M. Meerschaert and C. Tadjeran, Finite difference approximations for fractional
advection-dispersion, J. Comput. Appl. Math., 172 (2004), pp. 65-77.

\bibitem{mr98} R. Metzler,  T.F.Nonnenmacher. Fractional diffusion, waiting-time distributions, and Cattaneo-type
equations. Phys. Rev. E. 57 (1998), pp.  6409-6414.

\bibitem{ot1}  S. Momani and Z. Odibat, Comparison between the homotopy perturbation method and the
variational iteration method for linear fractional partial differential equations, Comput.
Math. Appl., 54 (2007), pp. 910-919.

\bibitem{AMO08} D. Murio, Implicite finite difference approximation for
time fractional diffusion equations, Comput. Math. Appl. 56 (2008)
1138-1145.

\bibitem{IP99} I. Podlubny, Fractional Differential Equations, vol. 198, Academic Press, San Diego,Calif, USA,1999.

\bibitem{ot2}I. Podlubny, A. Chechkin, T. Skovranek, Y. Q. Chen, and B. M. V. Jara, Matrix approach
to discrete fractional calculus II: Partial fractional differential equations, J. Comput. Phys.,
228 (2009), pp. 3137-3153.

\bibitem{fe2} J. P. Roop, Computational aspects of FEM approximation of fractional advection
dispersion equations on bounded domains in $R^2$, J. Comput. Appl. Math., 193 (2006), pp. 243-268.


\bibitem{ss07} S. S. Ray, Exact solutions for time-fractional diffusion-wave equations by decomposition method, Phys. Scr., 75(2007), 53.

\bibitem{lx} L. Shao, X. Feng, Y. He, The local discontinuous Galerkin finite element method for Burger's equation.
  Math. Comput. Modelling, 54 (2011), pp. 2943-2954.

  \bibitem{ot3} E. Sousa, A second order explicit finite difference method for the fractional advection
diffusion equation, Comput. Math. Appl., 64 (2012), pp. 3141-3152.

\bibitem{fd9} H. Wang, K. X. Wang, and T. Sircar, A direct $O(N log_2 N)$ finite difference method for
fractional diffusion equations, J. Comput. Phys., 229 (2010), pp. 8095-8104.

\bibitem{fd10}K. Wang and H. Wang, A fast characteristic finite difference method for fractional
advection-diffusion equations, Adv. Water Resour., 34 (2011), pp. 810-816.

\bibitem{wz08} L. Wang,  X. Zhou and X. Wei. Heat Conduction, Berlin:
Springer, 2008 .

\bibitem{xia}  Y. Xia, Y. Xu and C.-W. Shu, Application of the local discontinuous Galerkin method for the Allen-Cahn/Cahn-Hilliard system, Commun. Comput. Phys. 5 (2009),pp. 821-835.

\bibitem{xushu08} Y. Xu, C.-W. Shu, Local discontinuous Galerkin method for the
Camassa-Holm equation, SIAM J. Numer. Anal. 46 (2008), pp.
1998-2021.

\bibitem{jan1} Q. Xu, J.S. Hesthaven, Discontinuous Galerkin Method for Fractional Convection-Diffusion Equations. SIAM J. Numer. Anal.  52  (2014), pp. 405-423.

\bibitem{jan2} Q. Xu, J.S. Hesthaven, F. Chen,
A parareal method for time-fractional differential equations. J. Comput. Physics 293 (2015), pp. 173-183.


\bibitem{ot4}Q. Q. Yang, I. Turner, F. Liu, and M. Ilic, Novel numerical methods for solving the timespace
fractional diffusion equation in two dimensions, SIAM J. Sci. Comput., 33 (2011),
pp. 1159-1180.

\bibitem{Y10} A. Yildirim, He's homtopy perturbation method for solving the
space and time fractional telegraph equations, Int. J. Comput. Math.
87 (2010) 2998-3006.

\bibitem{fd11} S. B.Yuste, Weighted average finite difference methods for fractional diffusion equations,
J. Comput. Phys., 216 (2006), pp. 264-274.

\bibitem{fd12} P. Zhuang, F. Liu, V. Anh, and I. Turner, New solution and analytical techniques of the
implicit numerical method for the anomalous subdiffusion equation, SIAM J. Numer. Anal.,
46 (2008), pp. 1079-1095.

\bibitem{zhang} X. Zhang, B. Tang and  Y. He, Homotopy analysis method for higher-order fractional integro-differential equations, Comput. Math. Appl., 62(2011), pp. 3194-3203.

\bibitem{fd13} X. Zhao, Z.Z.Sun. Compact Crank-Nicolson Schemes for a Class
of Fractional Cattaneo Equation in Inhomogeneous
Medium. J. Sci. Comput. 62 (2015), pp. 747-771.

\bibitem{fe3} Y. Y. Zheng, C. P. Li, and Z. G. Zhao, A note on the finite element method for the
space fractional advection diffusion equation, Comput. Math. Appl., 59 (2010), pp. 1718-1726.

\bibitem{wei1}L.L. Wei, Y.N. He, Analysis of the fractional Kawahara equation using an implicit fully discrete local discontinuous Galerkin method, Numer. Methods Partial Differential Eq., 29 (2013), pp. 1441-1458.

\bibitem{wei2}L.L. Wei, Y.N. He, Analysis of a fully discrete local discontinuous Galerkin method for time-fractional fourth-order problems. Appl. Math. Model., 38 (2014), pp. 1511-1522.
\end{thebibliography}
\end{document}